\documentclass[12pt,reqno]{amsart} 
\usepackage{amssymb,amscd,url}

\begin{document}

\setlength{\textwidth}{13.5cm}


\newcommand{\DRAFTNUMBER}{2}
\newcommand{\DATE}{\today}
\newcommand{\TITLE}{A Local-Global Criterion for Dynamics on $\PP^1$}
\newcommand{\TITLERUNNING}{A Local-Global Criterion for Dynamics on $\PP^1$}


\hyphenation{ca-non-i-cal archi-me-dean non-archi-me-dean}


\newtheorem{theorem}{Theorem}
\newtheorem{lemma}[theorem]{Lemma}
\newtheorem{conjecture}[theorem]{Conjecture}
\newtheorem{question}[theorem]{Question}
\newtheorem{proposition}[theorem]{Proposition}
\newtheorem{corollary}[theorem]{Corollary}
\newtheorem*{claim}{Claim}

\theoremstyle{definition}
\newtheorem*{definition}{Definition}
\newtheorem*{Notation}{Notation}
\newtheorem{example}[theorem]{Example}

\theoremstyle{remark}
\newtheorem{remark}[theorem]{Remark}
\newtheorem*{acknowledgement}{Acknowledgements}



\newenvironment{notation}[0]{%
  \begin{list}%
    {}%
    {\setlength{\itemindent}{0pt}
     \setlength{\labelwidth}{2\parindent}
     \setlength{\labelsep}{\parindent}
     \setlength{\leftmargin}{3\parindent}
     \setlength{\itemsep}{0pt}
     }%
   }%
  {\end{list}}

\newenvironment{parts}[0]{%
  \begin{list}{}%
    {\setlength{\itemindent}{0pt}
     \setlength{\labelwidth}{1.5\parindent}
     \setlength{\labelsep}{.5\parindent}
     \setlength{\leftmargin}{2\parindent}
     \setlength{\itemsep}{0pt}
     }%
   }%
  {\end{list}}
\newcommand{\Part}[1]{\item[\upshape#1]}

\renewcommand{\a}{\alpha}
\renewcommand{\b}{\beta}
\newcommand{\g}{\gamma}
\renewcommand{\d}{\delta}
\newcommand{\e}{\epsilon}
\newcommand{\f}{\varphi}
\newcommand{\bfphi}{{\boldsymbol{\f}}}
\renewcommand{\l}{\lambda}
\renewcommand{\k}{\kappa}
\newcommand{\lhat}{\hat\lambda}
\newcommand{\m}{\mu}
\newcommand{\bfmu}{{\boldsymbol{\mu}}}
\renewcommand{\o}{\omega}
\renewcommand{\r}{\rho}
\newcommand{\rbar}{{\bar\rho}}
\newcommand{\s}{\sigma}
\newcommand{\sbar}{{\bar\sigma}}
\renewcommand{\t}{\tau}
\newcommand{\z}{\zeta}

\newcommand{\D}{\Delta}
\newcommand{\G}{\Gamma}
\newcommand{\F}{\Phi}

\newcommand{\ga}{{\mathfrak{a}}}
\newcommand{\gA}{{\mathfrak{A}}}
\newcommand{\gb}{{\mathfrak{b}}}
\newcommand{\gB}{{\mathfrak{B}}}
\newcommand{\gC}{{\mathfrak{C}}}
\newcommand{\gD}{{\mathfrak{D}}}
\newcommand{\gE}{{\mathfrak{E}}}
\newcommand{\gm}{{\mathfrak{m}}}
\newcommand{\gn}{{\mathfrak{n}}}
\newcommand{\go}{{\mathfrak{o}}}
\newcommand{\gO}{{\mathfrak{O}}}
\newcommand{\gp}{{\mathfrak{p}}}
\newcommand{\gP}{{\mathfrak{P}}}
\newcommand{\gq}{{\mathfrak{q}}}
\newcommand{\gQ}{{\mathfrak{Q}}}
\newcommand{\gR}{{\mathfrak{R}}}

\newcommand{\Abar}{{\bar A}}
\newcommand{\Ebar}{{\bar E}}
\newcommand{\Kbar}{{\bar K}}
\newcommand{\kbar}{{\bar k}}
\newcommand{\Pbar}{{\bar P}}
\newcommand{\Sbar}{{\bar S}}
\newcommand{\Tbar}{{\bar T}}
\newcommand{\ybar}{{\bar y}}
\newcommand{\phibar}{{\bar\f}}

\newcommand{\Acal}{{\mathcal A}}
\newcommand{\Bcal}{{\mathcal B}}
\newcommand{\Ccal}{{\mathcal C}}
\newcommand{\Dcal}{{\mathcal D}}
\newcommand{\Ecal}{{\mathcal E}}
\newcommand{\Fcal}{{\mathcal F}}
\newcommand{\Gcal}{{\mathcal G}}
\newcommand{\Hcal}{{\mathcal H}}
\newcommand{\Ical}{{\mathcal I}}
\newcommand{\Jcal}{{\mathcal J}}
\newcommand{\Kcal}{{\mathcal K}}
\newcommand{\Lcal}{{\mathcal L}}
\newcommand{\Mcal}{{\mathcal M}}
\newcommand{\Ncal}{{\mathcal N}}
\newcommand{\Ocal}{{\mathcal O}}
\newcommand{\Pcal}{{\mathcal P}}
\newcommand{\Qcal}{{\mathcal Q}}
\newcommand{\Rcal}{{\mathcal R}}
\newcommand{\Scal}{{\mathcal S}}
\newcommand{\Tcal}{{\mathcal T}}
\newcommand{\Ucal}{{\mathcal U}}
\newcommand{\Vcal}{{\mathcal V}}
\newcommand{\Wcal}{{\mathcal W}}
\newcommand{\Xcal}{{\mathcal X}}
\newcommand{\Ycal}{{\mathcal Y}}
\newcommand{\Zcal}{{\mathcal Z}}

\renewcommand{\AA}{\mathbb{A}}
\newcommand{\BB}{\mathbb{B}}
\newcommand{\CC}{\mathbb{C}}
\newcommand{\FF}{\mathbb{F}}
\newcommand{\GG}{\mathbb{G}}
\newcommand{\NN}{\mathbb{N}}
\newcommand{\PP}{\mathbb{P}}
\newcommand{\QQ}{\mathbb{Q}}
\newcommand{\RR}{\mathbb{R}}
\newcommand{\ZZ}{\mathbb{Z}}

\newcommand{\bfa}{{\mathbf a}}
\newcommand{\bfb}{{\mathbf b}}
\newcommand{\bfc}{{\mathbf c}}
\newcommand{\bfe}{{\mathbf e}}
\newcommand{\bff}{{\mathbf f}}
\newcommand{\bfg}{{\mathbf g}}
\newcommand{\bfp}{{\mathbf p}}
\newcommand{\bfr}{{\mathbf r}}
\newcommand{\bfs}{{\mathbf s}}
\newcommand{\bft}{{\mathbf t}}
\newcommand{\bfu}{{\mathbf u}}
\newcommand{\bfv}{{\mathbf v}}
\newcommand{\bfw}{{\mathbf w}}
\newcommand{\bfx}{{\mathbf x}}
\newcommand{\bfy}{{\mathbf y}}
\newcommand{\bfz}{{\mathbf z}}
\newcommand{\bfA}{{\mathbf A}}
\newcommand{\bfF}{{\mathbf F}}
\newcommand{\bfB}{{\mathbf B}}
\newcommand{\bfD}{{\mathbf D}}
\newcommand{\bfG}{{\mathbf G}}
\newcommand{\bfI}{{\mathbf I}}
\newcommand{\bfM}{{\mathbf M}}
\newcommand{\bfzero}{{\boldsymbol{0}}}

\newcommand{\Adele}{\textsf{\upshape A}}
\newcommand{\Aut}{\operatorname{Aut}}
\newcommand{\Br}{\operatorname{Br}}  
\newcommand{\Closure}{\textsf{\upshape C}} 
\newcommand{\Disc}{\operatorname{Disc}}
\newcommand{\density}{{\boldsymbol\delta}}
\newcommand{\densitysup}{\overline{\density}}
\newcommand{\densityinf}{\underline{\density}}
\newcommand{\Div}{\operatorname{Div}}
\newcommand{\End}{\operatorname{End}}
\newcommand{\Fbar}{{\bar{F}}}
\newcommand{\FOD}{\textup{FOM}}
\newcommand{\FOM}{\textup{FOD}}
\newcommand{\Gal}{\operatorname{Gal}}
\newcommand{\GL}{\operatorname{GL}}
\newcommand{\Index}{\operatorname{Index}}
\newcommand{\Image}{\operatorname{Image}}
\newcommand{\liftable}{{\textup{liftable}}}
\newcommand{\hhat}{{\hat h}}
\newcommand{\Ksep}{K^{\textup{sep}}}
\newcommand{\Ker}{{\operatorname{ker}}}
\newcommand{\Lsep}{L^{\textup{sep}}}
\newcommand{\Lift}{\operatorname{Lift}}
\newcommand{\pp}{\operatorname{pp}}  
\newcommand{\vlim}{\operatornamewithlimits{\text{$v$}-lim}}
\newcommand{\wlim}{\operatornamewithlimits{\text{$w$}-lim}}
\newcommand{\MOD}[1]{~(\textup{mod}~#1)}
\newcommand{\Norm}{{\operatorname{\mathsf{N}}}}
\newcommand{\notdivide}{\nmid}
\newcommand{\normalsubgroup}{\triangleleft}
\newcommand{\odd}{{\operatorname{odd}}}
\newcommand{\onto}{\twoheadrightarrow}
\newcommand{\Orbit}{\mathcal{O}}
\newcommand{\ord}{\operatorname{ord}}
\newcommand{\Per}{\operatorname{Per}}
\newcommand{\PrePer}{\operatorname{PrePer}}
\newcommand{\PGL}{\operatorname{PGL}}
\newcommand{\Pic}{\operatorname{Pic}}
\newcommand{\Prob}{\operatorname{Prob}}
\newcommand{\Qbar}{{\bar{\QQ}}}
\newcommand{\rank}{\operatorname{rank}}
\newcommand{\Resultant}{\operatorname{Res}}
\renewcommand{\setminus}{\smallsetminus}
\newcommand{\SL}{\operatorname{SL}}
\newcommand{\Span}{\operatorname{Span}}
\newcommand{\tors}{{\textup{tors}}}
\newcommand{\Trace}{\operatorname{Trace}}
\newcommand{\twistedtimes}{\mathbin{%
   \mbox{$\vrule height 6pt depth0pt width.5pt\hspace{-2.2pt}\times$}}}
\newcommand{\UHP}{{\mathfrak{h}}}    
\newcommand{\Wreath}{\operatorname{Wreath}}
\newcommand{\<}{\langle}
\renewcommand{\>}{\rangle}

\newcommand{\longhookrightarrow}{\lhook\joinrel\longrightarrow}
\newcommand{\longonto}{\relbar\joinrel\twoheadrightarrow}
\newcommand{\pmodintext}[1]{~\textup{(mod~$#1$)}}

\newcommand{\Spec}{\operatorname{Spec}}
\renewcommand{\div}{{\operatorname{div}}}

\newcounter{CaseCount}
\Alph{CaseCount}
\def\Case#1{\par\vspace{1\jot}\noindent
\stepcounter{CaseCount}
\framebox{Case \Alph{CaseCount}.\enspace#1}
\par\vspace{1\jot}\noindent\ignorespaces}

\title[\TITLERUNNING]{\TITLE}
\date{\DATE}

\author[Joseph H. Silverman]{Joseph H. Silverman}
\address{Mathematics Department, Box 1917, 151 Thayer Street,
         Brown University, Providence, RI 02912 USA}
\email{jhs@math.brown.edu}
\author[Jos\'e Felipe Voloch]{Jos\'e Felipe Voloch}
\address{Department of Mathematics, University of Texas,
1 University Station C1200, Austin, TX 78712 USA}
\email{voloch@math.utexas.edu}
\thanks{The first author's research supported by NSF DMS-0650017.}

\begin{abstract}
Let $K$ be a number field or a one-dimensional function field, let
$\f:\PP^1\to \PP^1$ be a rational map of degree at least two defined
over~$K$, let~$P\in \PP^1(K)$ be a point with infinite~$\f$-orbit, and
let~$Z\subset \PP^1$ be a finite set of points. We prove a
local-global criterion for the intersection of the $\f$-orbit of~$P$
and the finite set~$Z$.  This is a special case of a dynamical
Brauer-Manin criterion suggested by Hsia and Silverman.
\end{abstract}


\maketitle

\section*{Introduction}

A recent paper of Hsia and Silverman~\cite{hsiasilverman} describes an
analog of the Brauer--Manin obstruction in the setting of arithmetic
dynamics.  The definition is based on work of
Scharaschkin~\cite{scharaschkin}, who showed how to reformulate the
Brauer--Manin obstruction on curves of genus at least~$2$ as a purely
ad\`elic-geometric statement, with no reference to cohomology.  (See
also~\cite{volochlocalglobal}.)  In this note we prove a dynamical
analog of Scharaschkin's conjecture for dynamical systems
on~$\PP^1$. The proof uses a variety of tools, including a Zsigmondy
theorem for primitive divisors in dynamical systems that was recently
proven by Ingram and Silverman~\cite{IngSil}.

We recall from~\cite{hsiasilverman} the setup for the dynamical
version of the Brauer--Manin obstruction.

\begin{Notation}
\label{wanderingquestion}
Let~$K$ be a number field, let~$X/K$ be a projective variety, and let
\[
  \f:X\to X
\]
be a $K$-morphism of infinite order. Let~$\Adele_K$ denote the ring of
ad\`eles of~$K$, and for any point~$P\in X(K)$,
write~$\Closure\bigl(\Orbit_\f(P)\bigr)$ for the closure of the
orbit~$\Orbit_\f(P)$ of~$P$ in~$X(\Adele_K)$.  Let~$Z$ be a subvariety
of~$X$.
We clearly have an inclusion
\begin{equation}
  \label{BrMaInclusion}
  \Orbit_\f(P) \cap Z(K) \subset 
  \Closure\bigl(\Orbit_\f(P)\bigr) \cap Z(\Adele_K).
\end{equation}
The primary question raised in~\cite{hsiasilverman} is to ask when the
inclusion~\eqref{BrMaInclusion} is an equality.
\par
More generally, we can ask a similar question after excluding a (finite)
set of places. For any set of places~$S$ we let
\[
  \Adele'_{K,S}=\sideset{}{'}\prod_{v\notin S} K_v\subset\Adele_K,
\]
where the product is the usual adelic restricted product, and we write
$\Closure_S\bigl(\Orbit_\f(P)\bigr)$ for the closure of the orbit
in~$\Adele'_{K,S}$.
\end{Notation}

Our main result is a proof of a local-global
equality~\eqref{BrMaInclusion} for~$X=\PP^1$.

\begin{theorem}
\label{theorem:brmanobs}
Let~$K$ be a number field, let~$S$ be a finite set of places of~$K$,
and let~$\f:\PP^1\to\PP^1$ be a rational map of degree at least two
defined over~$K$. Let~$Z\subset\PP^1(K)$ be a finite set of
points. Then with notation as above,
\[
  \Orbit_\f(P) \cap Z(K) =
  \Closure_S\bigl(\Orbit_\f(P)\bigr) \cap Z(\Adele'_{K,S}).
\]
\end{theorem}

\begin{remark}
As noted in~\cite{hsiasilverman}, if~\eqref{BrMaInclusion} is to be an
equality, one should require that the subvariety~$Z$ contains no
positive dimensional $\f$-preperiodic subvarieties. However, since our
subvariety~$Z$ is a set of points, this condition on~$Z$ is vacuous.
\end{remark}

In Section~\ref{FunctionFieldAnalogs} we prove an analogous result for
function fields~$K/k$. If the map~$\f$ is not isotrivial, the proof is
essentially the same as the proof of Theorem~\ref{theorem:brmanobs},
but the isotrivial case must be handled separately.  Finally, in
Section~\ref{CommentsandSpeculations}, we make some further comments
and raise some related questions.

\section{Local-global dynamics on $\PP^1$ over number fields}
\label{section:P1unobstructed}

In this section we prove Theorem~\ref{theorem:brmanobs}.
We briefly recall some basic definitions from dynamical systems.

\begin{definition}
Let~$\f(z)\in K(z)$ be a rational function of degree~$d\ge2$, which we
may view as a morphism~\text{$\f:\PP^1_K\to\PP^1_K$}.  A
point~$\g\in\PP^1(\Kbar)$ is \emph{periodic for~$\f$} if~$\f^n(\g)=\g$
for some~$n\ge1$.  The smallest such~$n$ is called
the~\emph{$\f$-period of~$\g$}.  We say that~$\g$ is \emph{preperiodic
for~$\f$} if its \emph{$\f$-orbit}
\[
  \Orbit_\f(\g) =  \bigl\{\g,\f(\g),\f^2(\g),\dots\bigr\}
\]
is finite. Equivalently,~$\g$ is preperiodic
if some iterate~$\f^n(\g)$ is periodic.
\end{definition}

\begin{remark}
During the proof of Theorem~\ref{theorem:brmanobs}, we make use of the
following notation.  For any two points~$x,y\in\PP^1(K)$ and any prime
ideal~$\gp$ of~$K$, we write
\begin{equation}
  \label{x=ymodp}
  x\equiv y\pmod{\gp}
\end{equation}
to indicate that the reductions of~$x$ and~$y$ modulo~$\gp$ coincide.
Formally, this notation means that if we write~$x=[x_1,x_2]$
and~$y=[y_1,y_2]$ using $\gp$-integral coordinates such that at least
one coordinate is a~$\gp$-unit, then
\[
  x_1y_2-x_2y_1\equiv 0\pmod{\gp}.
\]
\par
Alternatively, we define a (nonarchimedean) chordal metric
on~$\PP^1$ for the absolute value~$v$ by (cf.\ \cite[\S2.1]{silverman:ads})
\[
  \D_v(x,y) = \frac{|x_1y_2-x_1y_1|_v}
    {\max\bigl\{|x_1|_v,|x_2|_v\bigr\}\cdot\max\bigl\{|y_1|_v,|y_2|_v\bigr\}}.
\]
This is independent of the choice of homogeneous coordinates for~$x$
and~$y$, and we define the congruence~\eqref{x=ymodp} to
mean~$\D_v(x,y)<1$, where~$v$ is the absolute value associated to the
prime~$\gp$.  More generally, we write
\[
  \prod (\a_i-\b_i) \equiv 0\pmod{\gp}
\]
for points~$\a_i,\b_i\in\PP^1(K)$ to mean that~$\prod\D_v(\a_i,\b_i)<1$.
\end{remark}

\begin{proof}[Proof of Theorem~$\ref{theorem:brmanobs}$]
Without loss of generality, we may assume that the set~$S$ contains all
archimedean places of~$K$ and all nonarchimedean places at which~$\f$
has bad reduction (see~\cite[\S2.5]{silverman:ads}).
Let~$\b\in\PP^1(K)$. If $\b$ is preperiodic, then its
orbit~$\Orbit_\f(\b)$ is finite, so the orbit is discrete
in~$\PP^1(K_v)$ for every place~$v$. Hence in this case we
have~$\Closure\bigl(\Orbit_\f(\b)\bigr)=\Orbit_\f(\b)$, so the
equality
\[
  \Orbit_\f(\b) \cap Z(K) 
    = \Closure\bigl(\Orbit_\f(\b)\bigr) \cap Z(\Adele_K).
\]
is obvious.

We suppose now that~$\b\in\PP^1(K)$ is a point with infinite
$\f$-orbit.  Since~$Z$ is a finite set, there can be only finitely
many iterates~$\f^n(\b)$ that lie in~$Z$, so replacing~$\b$ with
some~$\f^n(\b)$, we may assume that \text{$\Orbit_\f(\b)\cap
Z(K)=\emptyset$}. Under this assumption, we suppose that there exists a
point
\[
  \a\in\Closure_S\bigl(\Orbit_\f(\b)\bigr) \cap Z(\Adele'_{K,S})
\]
and derive a contradiction. The supposed existence of~$\a$
means that there is an increasing sequence of positive integers
\[
  \Ncal = \{n_1,n_2,n_3,\ldots\}
\]
such that the sequence of points~$\f^n(\b)$ with~$n\in\Ncal$ converges
adelically to~$\a\in Z(\Adele'_{K,S})$. Looking at each coordinate of
the adele~$\a$, we see in particular that for every prime~$\gp\notin
S$ of~$K$ there exists an integer~$N_\gp$ such that
\[
  \f^n(\b) \equiv \a_\gp \pmod{\gp}
  \qquad\text{for all $n\in\Ncal$ with $n\ge N_\gp$.}
\]
Note that since~$Z$ is a finite subset of~$\PP^1(K)$, each~$\a_\gp$ is
a point in~$\PP^1(K)$ considered as a subset of~$\PP^1(K_\gp)$.

We are going to use the dynamical Zsigmondy theorem from~\cite{IngSil},
whose statement we recall after making one definition.

\begin{definition}
Let~$\f:\PP^1\to\PP^1$ and let~$\g\in\PP^1$. We say that~$\f$ is of
\emph{polynomial type at~$\g$} if there is some~$k\ge1$ such that~$\g$
is a totally ramified fixed point of~$\f^k$. 
\end{definition}

\begin{remark}
\label{remark:poltype}
It is easy to check that~$\f$ is of polynomial type at~$\g$ if and
only if, when we conjugate~$\f$ by~$(z-\g)^{-1}$, the map~$\f^k$ is
conjugated to a polynomial in~$z$.  We also remark that it is an
exercise using the Riemann-Hurwitz genus formula to show that if~$\f$
is of polynomial type at~$\g$, then the~$\f$-period of~$\g$ is at
most~$2$, cf.~\cite[Theorem~1.7]{silverman:ads}.
\end{remark}

\begin{theorem}
\label{theorem:dynzsig}
\textup{(Ingram--Silverman \cite{IngSil})}\enspace Let~$K$ be a number
field, let~$S$ be a finite set of primes, let~$\b\in\PP^1(K)$ be a
point with infinite orbit, and let~$\g\in\PP^1(K)$ be a preperiodic
point.  Let~$\f:\PP^1\to\PP^1$ be a $K$-morphism of degree at least
two that is not of polynomial type at~$\g$. There exists an
integer~$M=M_{\f,\b,\g}$ so that for all~$m\ge M$ there exists a prime
ideal~$\gq_m\notin S$ of~$K$ satisfying
\begin{align*}
  \f^m(\b)&\equiv\g\pmod{\gq_m}, \\
  \f^i(\b)&\not\equiv\g\pmod{\gq_m}
  \quad\text{for all $0\le i<m$.}
\end{align*}
\end{theorem}

In the terminology of~\cite{IngSil}, the prime ideal~$\gp_m$ is a
\emph{primitive divisor}
for~$\f^m(\b)-\g$. Theorem~\ref{theorem:dynzsig} says
that~$\f^m(\b)-\g$ has a primitive divisor for all sufficiently large
values of~$m$.  There are many classical theorems showing the
existence of primitive divisors for the multiplicative group and for
elliptic curves, see for example~\cite{MR0344221,MR961918,MR1546236}.
The proof of Theorem~\ref{theorem:dynzsig} uses a dynamical
analog~\cite{MR1240603} of Siegel's theorem~\cite[IX.3.1]{MR1329092}
for integral points on elliptic curves.

\begin{remark}
The theorem stated in~\cite{IngSil} does not exclude a finite set of
primes, but since~$S$ is finite, we immediately get
Theorem~\ref{theorem:dynzsig} from~\cite{IngSil} by increasing~$M$ so
as to exclude the finitely many~$m$ such that the primitive divisor
of~$\f^m(\b)-\g$ is in~$S$.
\end{remark}

We fix an integer~$R\ge3$ and extend the field~$K$ so that all of
the periodic points of~$\f$ of exact period~$R$ are in~$\PP^1(K)$.
Let~$\g\in\PP^1(K)$ be a point of exact period~$R$ for~$\f$. As noted
in Remark~\ref{remark:poltype}, the map~$\f$ is not of polynomial type
at~$\g$, so we can apply Theorem~\ref{theorem:dynzsig} to find an
integer~$M=M_{\f,\b,\g}$ such that~$\f^m(\b)-\g$ has a non-$S$ primitive
divisor for all~$m\ge M$. Doing this for each of the finitely many
points of period~$R$, we may choose one~$M=M_{\f,\b,R}$ that works for
all~$\g$ of period~$R$. 

Next we fix an integer~$T\ge1$ and we choose a non-$S$ primitive
divisor for~$\f^m(\b)-\g$ for each~$m$ between~$M$ and~$M+T$. Thus we
find \emph{distinct} prime
ideals~$\gq_{M,\g},\gq_{M+1,\g},\ldots,\gq_{M+T,\g}\notin S$ such
that
\[
  \f^m(\b)\equiv\g\pmod{\gq_{m,\g}}
  \quad\text{for~$M\le m\le M+T$.}
\]
We have indicated the dependence of~$\gq_{m,\g}$ on both~$m$ and~$\g$,
because we now replace~$\g$ by~$\f^i\g$ for each~$1\le i<R$ and apply
the same argument. Thus we find prime ideals~$\gq_{m,\f^i\g}\notin S$
satisfying
\[
  \f^m(\b)\equiv\f^i(\g)\pmod{\gq_{m,\f^i\g}}
  \quad\text{for~$M\le m\le M+T$ and $0\le i<R$.}
\]
Further, for any particular value of~$i$, the ideals~$\gq_{m,\f^i\g}$
are distinct for~$M\le m\le M+T$. (These ideals also depend on~$\f$
and~$\b$, of course, but~$\f$ and~$\b$ are fixed throughout the proof.)
This gives us a finite set of non-$S$ primes
\[
  \Qcal = \Qcal_{R,T} = \bigl\{\gq_{m,\f^i\g} : 
    \text{$M\le m\le M+T$ and $0\le i<R$} \bigr\}.
\]
(Note that~$M=M_{\f,\b,R}$ is independent of~$T$.)

Consider the product
\begin{equation}
  \label{eqn:prodfmbfrg}
  \prod_{r=0}^{R-1} \bigl(\f^m(\b)-\f^r(\g)\bigr).
\end{equation}
It is divisible by each of the (not necessarily distinct) prime
ideals
\begin{equation}
  \label{eqn:gqlist}
  \gq_{m,\g}, \gq_{m,\f\g}, \gq_{m,\f^2\g},\cdots,
   \gq_{m,\f^{R-1}\g}.
\end{equation}
Further, the fact that~$\g$ is periodic with period~$R$ implies
that for any~$n\ge m$, the product
\begin{equation}
  \label{eqn:prodfnbfrg}
  \prod_{r=0}^{R-1} \bigl(\f^n(\b)-\f^r(\g)\bigr).
\end{equation}
is divisible by all of the primes in the list~\eqref{eqn:gqlist}.
To see this, let~$\gq$ be any prime listed in~\eqref{eqn:gqlist}.
Then we use the $R$-periodicity of~$\g$ to write
\[
  \prod_{r=0}^{R-1} \bigl(\f^n(\b)-\f^r(\g)\bigr)
  = \prod_{r=0}^{R-1} \bigl(\f^n(\b)-\f^{r+1}(\g)\bigr),
\]
and this last product is clearly divisible by
\[
  \prod_{r=0}^{R-1} \bigl(\f^{n-1}(\b)-\f^r(\g)\bigr).
\]
(We are using the general fact that~$\f(b)-\f(a)$ is divisible
by~$b-a$.)  Now a downward induction shows that~\eqref{eqn:prodfnbfrg}
is divisible by~\eqref{eqn:prodfmbfrg}.

Recall the sequence of positive integer~$\Ncal$ such that~$\f^n(\b)$
converges $\Adele'_{K,S}$- adelically to~$\a$ as~$n\to\infty$
with~$n\in\Ncal$.  In particular, for each $\gq\in\Qcal$ we have
\begin{equation}
  \label{eqn:fnb=aqmodq}
  \f^n(\b)\equiv \a_\gq\pmod{\gq}
\end{equation}
for all sufficiently large~$n\in\Ncal$.  
(Note that $\Qcal\cap S=\emptyset$.)
We recall that each~$\a_\gq$ is chosen from the finite set~$Z$, so we
can reformulate~\eqref{eqn:fnb=aqmodq} as
\begin{equation}
  \label{eqn:prodfnbaqmodq}
  \prod_{a\in Z} \bigl(\f^n(\b)-a\bigr)\equiv 0\pmod{\gq}
\end{equation}
for sufficiently large~$n\in\Ncal$ and for every~$\gq\in\Qcal$.

For notational convenience, we define
\[
  \gQ_{R,T} = \prod_{\gq\in\Qcal_{R,T}} \gq
  = \operatorname{Radical}
     \left(\prod_{M\le m\le M+T}\prod_{0\le i<R} \gq_{m,\f^i\g}\right).
\]
The fact that the ideals~$\gq_{m,\f^i\g}$ are distinct for fixed~$i$
and varying~$m$ implies that
\[
  \#\Qcal_{R,T} > T,
\]
so~$\gQ_{R,T}$ is a product of more than~$T$ distinct prime ideals.

With this notation, we can rewrite~\eqref{eqn:prodfnbaqmodq} as
\begin{equation}
  \label{eqn:prodfnbaqmodq1}
  \prod_{a\in Z} \bigl(\f^n(\b)-a\bigr)\equiv 0\pmod{\gQ_{R,T}}
  \qquad\text{for sufficiently large $n\in\Ncal$.}
\end{equation}
Similarly, it follows from our earlier discussion that
\begin{equation}
  \label{eqn:Rfnbrg}
  \prod_{r=0}^{R-1} \bigl(\f^n(\b)-\f^r(\g)\bigr)
  \equiv 0\pmod{\gQ_{R,T}}
  \qquad\text{for all $n\ge M_R$.}
\end{equation}
Comparing~\eqref{eqn:prodfnbaqmodq1} and~\eqref{eqn:Rfnbrg}
for any single sufficiently large value of~$n\in\Ncal$,
we conclude that
\begin{equation}
  \label{eqn:rRaZ}
  \prod_{r=0}^{R-1} \prod_{a\in Z} \bigl(\f^r(\g)-a\bigr)
  \equiv 0\pmod{\gQ_{R,T}}.
\end{equation}

This last congruence is very interesting, because the product on 
the left-hand side does not depend on~$T$, while the ideal~$\gQ_{R,T}$
is a product of more than~$T$ distinct prime ideals. Letting~$T\to\infty$,
we conclude that the product in~\eqref{eqn:rRaZ} vanishes. Therefore 
there is at least one value~$r_0$ and at least one point~$a_0\in Z$ such
that
\[
  \f^{r_0}(\g) = a_0.
\]
To recapitulate, we have proven that given any point~$\g\in\PP^1$ that
is periodic for~$\f$ of period at least~$3$, there is some point in
the $\f$-orbit of~$\g$ that lies in the set~$Z$. But a rational map
on~$\PP^1$ has periodic points of infinitely many periods (more
precisely, it has a periodic point of exact period~$R$ for
every~$R\ge5$, see~\cite{MR1128089}), which contradicts the
assumption that~$Z$ is a finite set.
\end{proof} 

\section{Local-global dynamics on $\PP^1$ over function fields}
\label{FunctionFieldAnalogs}

In this section we prove the analogue of
Theorem~\ref{theorem:brmanobs} for function fields.  Let~$k$ be an
algebraically closed field and let~$K/k$ be a function field, that
is,~$K$ is a finitely generated extension of~$k$ of transcendence
degree one.  We employ the same notation as in the introduction except
that we consider only the places of~$K/k$, i.e., the absolute values
of~$K$ that are trivial on~$k$.

\begin{theorem}
\label{theorem:brmanobsf}
Let~$K/k$ be a function field as above, let~$S$ be a finite set of
places of~$K/k$, and let~$\f:\PP^1\to\PP^1$ be a rational map of
degree at least two defined over~$K$. Let~$Z\subset\PP^1(K)$ be a
finite set of points. Then with notation as in the introduction,
\[
  \Orbit_\f(P) \cap Z(K) =
  \Closure_S\bigl(\Orbit_\f(P)\bigr) \cap Z(\Adele'_{K,S}).
\]
\end{theorem}

\begin{proof}
A rational map~$\f$ defined over~$K$ is isotrivial (or split) if there
exists a M\"obius transformation~$M \in \PGL_2(\Kbar)$ such that
$M\circ \f \circ M^{-1}$ has coefficients in~$k$.  As pointed out by
T. Tucker~\cite[Remark~4]{IngSil}, the Zsigmondy theorem
in~\cite{IngSil} can be proved for non-isotrivial rational maps over
function fields using the results of Baker~\cite{Baker}. We observe
that the characteristic zero assumption is not used in~\cite{Baker},
and therefore is not needed to extend the results of~\cite{IngSil},
contrary to what is stated there. Thus the argument used to prove
Theorem~\ref{theorem:brmanobs} applies verbatim in the function field
case if~$\f$ is non-isotrivial.

Assume now that~$\f$ is isotrivial. There is no loss of generality in
extending~$K$ by adjoining the coefficients of~$M$, and then
replacing~$\f$ by $M\circ \f \circ M^{-1}$, we may assume that~$\f$
has coefficients in~$k$.

Arguing again as in the proof of Theorem~\ref{theorem:brmanobs}, the
case of preperiodic points is straightforward. 

We suppose now that~$\b\in\PP^1(K)$ is a point with infinite
$\f$-orbit.  Since~$Z$ is a finite set, there can be only finitely
many iterates~$\f^n(\b)$ that lie in~$Z$, so replacing~$\b$ with
some~$\f^n(\b)$, we may assume that \text{$\Orbit_\f(\b)\cap
Z=\emptyset$}. Under this assumption, we suppose that there exists a
point
\[
  \a\in\Closure_S\bigl(\Orbit_\f(\b)\bigr) \cap Z(\Adele'_{K,S})
\]
and derive a contradiction. The supposed existence of~$\a$
means that there is an increasing sequence of positive integers
\[
  \Ncal = \{n_1,n_2,n_3,\ldots\}
\]
such that the sequence of points~$\f^n(\b)$ with~$n\in\Ncal$ converges
adelically to~$\a\in Z(\Adele'_{K,S})$. 

We note first that~$\b \not\in k$, since otherwise the sequence 
$\f^{n_i}(\b) \in k$ could not converge~$\gp$-adically for any place~$\gp$
of~$K/k$ unless it became eventually constant, which would force~$\b$ to be
preperiodic, contrary to our current assumption.

Assume now that~$\f$ is not purely inseparable. Choose an
element \text{$b \in k$} that is periodic with respect to~$\f$ and
such that the backwards orbit of~$b$ with respect to~$\f$ is infinite
and such that no element of~$Z$ is in the (forward) orbit
of~$b$. It follows that the set~$T$ of places~$\gp$ of~$K/k$ for which an
element of~$Z$ is congruent modulo~$\gp$ to an element of the
(forward) orbit of~$b$ is finite. Now choose an element~$a \in k$ in
the backwards orbit of~$b$ with respect to~$\f$, such that~$a$ is not
in the (forward) orbit of~$b$ and such that the places of~$K/k$
extending the place~$\b - a$ of~$k(\b)$ do not belong
to~$T$. Let~$\gp$ be a place of~$K/k$ extending the place~$\b - a$
of~$k(\b)$. By assumption, there exists~$\a_{\gp} \in Z(K)$
with~$\f^{n_i}(\b)$ converging~$\gp$-adically to~$\a_{\gp}$. Thus, for
all large~$i$, we have
\[
\f^{n_i}(a) \equiv \f^{n_i}(\b) \equiv \a_{\gp} \pmod{\gp}.
\]
But, for~$i$ large, $\f^{n_i}(a)$ belongs to the forward orbit
of~$b$. This forces \text{$\gp \in T$}, which is a contradiction.

Finally, we have to deal with the case that~$\f$ is purely inseparable.
Without loss of generality, we may assume that~$\f(z) = z^q$,where~$q$
is a power of the characteristic of~$K$.  Let~$a\in k$ and let~$\gp$
be a place of~$K/k$ extending the place~$\b - a$ of~$k(\b)$. Then the
only way that a sequence $\f^{n_i}(\b)$ can converge~$\gp$-adically is
if~$a$ is in a finite field, and in this case converges to an element
in the Frobenius orbit of~$a$.  There are only finitely many such~$a$
whose orbits intersect the finite set~$Z$, and hence~$\f^{n_i}(\b)$
cannot converge adelically. This completes the proof.
\end{proof}

\section{Comments and Speculations}
\label{CommentsandSpeculations}

\begin{remark}
We have stated our principal results (Theorems~\ref{theorem:brmanobs}
and~\ref{theorem:brmanobsf}) for rational maps of degree at least two,
but we can also consider maps of degree one. After a finite extension
of the field~$K$, any map \text{$\f:\PP^1\to\PP^1$} of degree one may
be conjugated to a map either of the form \text{$\f(z)=az$} or of the
form~$\f(z)=z+1$. There are three cases to consider.
\begin{enumerate}
\item
If~$\f(z)=az$ and~$a$ is a root of unity, say~$a^n=1$,
then~\text{$\f^n(z)=z$} is the identity map, so every point has finite
orbit. Hence the closure of~$\Orbit_\f(P)$ is equal to~$\Orbit_\f(P)$,
so the conclusions of Theorems~\ref{theorem:brmanobs}
and~\ref{theorem:brmanobsf} are clearly true.
\item
If~$\f(z)=az$ and~$a$ is not a root of unity, then we can mimic the
proof of Theorem~\ref{theorem:brmanobs}, replacing
Theorem~\ref{theorem:dynzsig} with the classical Zsigmondy theorem for
the multiplicative group.  So Theorems~\ref{theorem:brmanobs}
and~\ref{theorem:brmanobsf} are also true in this case.
\item
If~$\f(z)=z+1$, then Theorem~\ref{theorem:brmanobs} is not true in
general. For example, take~$Z=\{0,\infty\}$ and~$\b=1$. Then the orbit
of~$\b$ is \text{$\Orbit_\f(\b)=\{1,2,3,\ldots\}$}, which does not
intersect~$Z$. On the other hand, the orbit of~$\b$ contains the
subsequence \text{$\{n!:n=1,2,3,\dots\}$}, and this subsequence converges to
the point in~$Z(\Adele_\QQ)$ that is~$\infty$ at the archimedean place
and is~$0$ at all nonarchimedean places.
Hence~\text{$\Closure\bigl(\Orbit_\f(\b)\bigr)\cap
Z(\Adele_\QQ)\ne\emptyset$}.
\end{enumerate}

\end{remark}

\begin{remark}
Let~$K/\QQ$ be a number field, let~$f(z)\in K(z)$ be a rational
function of degree at least two, and let $\f(z)=z-f(z)/f'(z)$. Iteration
of~$\f(z)$ is Newton's method of finding a root of~$f(z)$. Let~$\a\in
K$ be a point whose orbit~$\Orbit_\f(\a)$ does not contain a root
of~$f$, e.g., any point that is not preperiodic for~$\f$. Then
Theorem~\ref{theorem:brmanobs} says that there are infinitely many
places~$v\in S$ such that Newton's method applied to~$\a$ does not
converge to a root of~$f(z)$.
\par
We raise the following question: In this situation (with $\deg
f(z)\ge3$), is it true that Newton iteration applied to~$\a$
converges for infinitely many places and also does not converge for
infinitely many places?
\end{remark}

\begin{remark}
We observe that an equality
\[
  \Orbit_\f(P) \cap Z(K) =
  \Closure\bigl(\Orbit_\f(P)\bigr) \cap Z(\Adele_K),
\]
such as given in~\cite{hsiasilverman} and in
Theorem~\ref{theorem:brmanobs}, provides an algorithm for determining
if~$\Orbit_\f(P) \cap Z(K)$ is nonempty.  Thus ``by day'' one computes
elements of the orbit and checks if they are in~$Z$, while ``by
night'' one checks if there are points in $\Orbit_\f(P)\cap Z$
modulo~$\gQ$ for ideals~$\gQ$ that are more and more divisible by
primes not in~$S$. For a fixed~$\gQ$ this is a finite
computation. This is analogous to the fact that the Scharaschkin
conjecture~\cite{scharaschkin} gives an algorithm to decide whether a
curve of genus at least~$2$ has any rational points (but not necessarily
to find them all).
\end{remark}

\begin{acknowledgement}
The authors would like to thank the organizers of the TateFest (May
2008), at which event this research was initiated.
\end{acknowledgement}




\end{document}